\documentclass[a4paper,13pt,oneside,onecolumn,preprint,sort&compress]{elsarticle}
\usepackage{amscd}
\usepackage{bm}	
\usepackage{setspace}			
\usepackage{mathtools}		
\usepackage{booktabs}		
\usepackage{float}			
\usepackage{amssymb}
\usepackage{mathrsfs}
\usepackage{amsthm}
\usepackage{dsfont}
\usepackage{mathrsfs}
\usepackage{graphicx}
\usepackage{epstopdf}
\usepackage{subfigure}
\usepackage{color}
\usepackage{latexsym}
\usepackage{amsmath}
\usepackage{amsfonts}
\usepackage{amssymb}
\usepackage{epsf}
\usepackage{graphicx}
\usepackage{epsf}
\usepackage{stmaryrd}
\usepackage{multicol}
\usepackage{subfigure}
\usepackage{ifpdf}
\usepackage{amsmath}
\usepackage{bm}				
\usepackage{mathtools}		
\usepackage{booktabs}		
\usepackage{float}			
\usepackage{amsfonts}
\usepackage{amsthm}
\usepackage{mathrsfs}
\usepackage{color}
\usepackage{graphicx}
\usepackage{CJK}
\usepackage{multirow}
\usepackage{graphicx,natbib,amssymb,lineno}
\usepackage[justification=centering]{caption}
\usepackage{bm}
\usepackage{diagbox}
\usepackage{makecell}
\usepackage{algorithmicx}
\usepackage[ruled]{algorithm}
\usepackage{algcompatible}
\usepackage{algpseudocode}
\graphicspath{{figure/}}
\ifpdf
\usepackage[%
  pdftitle={Instructions for use of the document class
    elsart},%
  pdfauthor={},%
  pdfsubject={The preprint document class elsart},%
  pdfkeywords={instructions for use, elsart, document class},%
  pdfstartview=FitH,%
  bookmarks=true,%
  bookmarksopen=true,%
  breaklinks=true,%
  colorlinks=true,%
  linkcolor=blue,anchorcolor=blue,%
  citecolor=blue,filecolor=blue,%
  menucolor=blue,pagecolor=blue,%
  urlcolor=blue]{hyperref}
\else
\usepackage[%
  breaklinks=true,%
  colorlinks=true,%
  linkcolor=blue,anchorcolor=blue,%
  citecolor=blue,filecolor=blue,%
  menucolor=blue,pagecolor=blue,%
  urlcolor=blue]{hyperref}
\fi
\makeatletter
\def\elsartstyle{%
    \def\normalsize{\@setfontsize\normalsize\@xiipt{14.5}}
    \def\small{\@setfontsize\small\@xipt{13.6}}
    \let\footnotesize=\small
    \def\large{\@setfontsize\large\@xivpt{18}}
    \def\Large{\@setfontsize\Large\@xviipt{22}}
    \skip\@mpfootins = 18\p@ \@plus 2\p@
    \normalsize
} \@ifundefined{square}{}{} \makeatother

\newtheorem{theorem}{Theorem}
\newtheorem{lemma}[theorem]{Lemma}

\theoremstyle{definition}

\newcommand{\ba}{\begin{array} }
\newcommand{\ea}{\end{array} }
\newcommand{\bae}{\begin{eqnarray}}
\newcommand{\eae}{\end{eqnarray}}
\newcommand{\bea}{\begin{eqnarray*}}
\newcommand{\eea}{\end{eqnarray*}}
\newcommand{\be}{\begin{equation}}
\newcommand{\ee}{\end{equation}}

\numberwithin{equation}{section}

\numberwithin{theorem}{section}
\numberwithin{Lemma}{section}
\usepackage[text={7.10in,9.95in},centering]{geometry}%这里设置每一页内容多少

\begin{document}
\begin{frontmatter}
\title{Meshfree method for solving the elliptic Monge-Amp$\grave{\textrm{e}}$re\\ equation with Dirichlet boundary\tnoteref{cor1}}
\author{Zhiyong Liu}
\author {Qiuyan Xu$^{*}$}
\ead{qiuyanxu@nxu.edu.cn} \cortext[cor1]{Corresponding author. }
\address{School of Mathematics and Statistics, \\Ningxia University,
Yinchuan, China, 750021 }
\tnotetext[t1]{The research is supported by the Natural Science Foundations of China (No.12061057,12202219),
and the Natural Science Foundations of Ningxia Province (No.2023AAC05001,2023AAC03003).}

\begin{abstract}
We prove the convergence of meshfree method for solving the elliptic Monge-Amp\textrm{$\grave{\textrm{e}}$}re equation with Dirichlet boundary on the bounded domain. $L^{2}$ error is obtained based on the kernel-based trial spaces generated by the compactly
supported radial basis functions. We obtain the convergence result when the testing discretization is finer than the trial discretization.
The convergence rate depend on the regularity of the solution, the smoothness of the computing domain,
and the approximation of scaled kernel-based spaces. The presented convergence theory covers a wide range of kernel-based trial spaces including stationary approximation and non-stationary approximation. An extension to non-Dirichlet boundary condition is in a forthcoming paper.
\end{abstract}
\end{frontmatter}
\thispagestyle{empty}
\section{Introduction}
We consider the meshfree method for numerically solving the Monge-Amp$\grave{\textrm{e}}$re equation with Dirichlet boundary
condition:
\begin{eqnarray}
  \det(D^{2}u(\textbf{x})) &=& f(\textbf{x}), \quad\textrm{in}\ \Omega\subset \mathbb{R}^{d},\label{ma1}\\
  u(\textbf{x}) &=& g(\textbf{x}),\quad \textrm{on}\ \partial\Omega,\label{ma2}
\end{eqnarray}
here $\Omega$ is a strictly convex polygonal domain with a smooth boundary $\partial\Omega$, $f$ is
a strictly positive function, and $\det(D^{2}u)$ is the determinant of the Hessian matrix $D^{2}u$. This equation
is usually augmented by the convexity constraint of $u$, which implies that the equation is elliptic.
The numerical solution of Monge-Amp$\grave{\textrm{e}}$re equation is an very active topic of research.
Many researchers developed kinds of numerical methods and analyzed their convergence, such as the finite
difference methods \cite{OBERMAN,BEN}, finite element methods \cite{FENG,BRENNER}, and spline methods \cite{LAI}.  In order to design a certain type of
numerical method which can avoid the tedious mesh generation and the domain integration and can be capable of
dealing with any irregular distribution of nodes, some
meshfree methods based radial basis functions have been researched, which is started from 2013 in the paper \cite{LIU12}.
And from then on, a series of meshfree methods have been proposed for solving the elliptic Monge-Amp$\grave{\textrm{e}}$re equation \cite{LIU13,LIU14,b19,LIU20}.
Unfortunately, the convergence theory of meshfree method has not been proven yet. The present paper will prove the convergence of meshfree
method for the Monge-Amp$\grave{\textrm{e}}$re equation. To simplify the proof, we only consider the case of taking $d=2$:
\begin{eqnarray}
  Fu:=\det(D^{2}u(x,y)) &=& f(x,y), \quad\textrm{in}\ \Omega\subset \mathbb{R}^{2},\label{2ma1}\\
  u(x,y) &=& g(x,y),\quad \textrm{on}\ \partial\Omega.\label{2ma2}
\end{eqnarray}
We will rely the B$\ddot{\textrm{o}}$hmer/Schaback's
nonlinear discretization theory \cite{b10,b13} and the ideas of the latest paper \cite{LIU23}, which discussed the convergence of unsymmetric radial basis functions method
for second order quasilinear elliptic equations.

The paper is organized as follows. Section 2 is devoted to some notation and an introduction of kernel-based trial spaces generated by compactly supported radial basis functions. Section 3 introduces
the well-posed theory of the Monge-Amp$\grave{\textrm{e}}$re equation and its meshfree discretization.
Section 4 is the theoretical analysis of convergence.
\section{Radial basis functions trial spaces}
A translation-invariant kernel such as the strictly positive definite radial function $\Phi(\textbf{x},\textbf{y})=\Phi(\textbf{x}-\textbf{y})$
will generate a reproducing kernel Hilbert space, the so-called native space $\mathcal{N}_{\Phi}(\Omega)$.
But we will make use of the native space $\mathcal{N}_{\Phi}(\mathbb{R}^{d})$ which defined in $\mathbb{R}^{d}$ and
described by Fourier transform. $\mathcal{N}_{\Phi}(\mathbb{R}^{d})$ consists of all functions $f\in L^{2}(\mathbb{R}^{d})$
with
\begin{equation}\label{nnorm}
\|f\|_{\Phi}^{2}=\int_{\mathbb{R}^{d}}\frac{|\widehat{f}(\boldsymbol{\omega})|^{2}}
{\widehat{\Phi}(\boldsymbol{\omega})}d\boldsymbol{\omega}<\infty.
\end{equation}
Hence, if the Fourier transform of $\Phi:\mathbb{R}^{d}\rightarrow\mathbb{R}$ has the algebraic decay condition
\begin{equation}\label{adc}
c_{1}(1+\|\boldsymbol{\omega}\|_{2}^{2})^{-\sigma}\leq \widehat{\Phi}(\boldsymbol{\omega})\leq c_{2}(1+\|\boldsymbol{\omega}\|_{2}^{2})^{-\sigma}
\end{equation}
for two fixed constants $0<c_{1}<c_{2}$, then the native space $\mathcal{N}_{\Phi}(\mathbb{R}^{d})$ will be norm equivalent to
the Sobolev space $W^{\sigma,2}(\mathbb{R}^{d})$. Thus, with a radial function $\Phi$ satisfied (\ref{adc}) and a scaling parameter
$\delta\in(0,1]$ we can construct a scaled radial function
\begin{equation}\label{rbf}
\Phi_{\delta}:=\delta^{-d}\Phi((\textbf{x}-\textbf{y})/\delta), \quad \textbf{x},\textbf{y}\in\mathbb{R}^{d}.
\end{equation}

Thus we can consider the finite dimensional kernel-based trial space generated by $\Phi_{\delta}$. Let
$Y^{I}=\{\textbf{y}_{1},\textbf{y}_{2},\ldots,\textbf{y}_{N_{I}}\}\subseteq \Omega$ and $Y^{B}=\{\textbf{y}_{N_{I}+1},\textbf{y}_{N_{I}+2},\ldots,\textbf{y}_{N}\}\subseteq \partial\Omega$
be some data sites which are arranged on the domain and its boundary respectively. Then the whole of centers are
$Y=Y^{I}\cup Y^{B}$, which can used to construct a trial space and determine its degrees of freedom.
We denote $B=B(\textbf{0},1)$ as the unit ball in $\mathbb{R}^{d-1}$ defined by using a number of $C^{k,\gamma}-$diffeomorphisms
$\psi_{j}:B\rightarrow V_{j}, \quad 1\leq j \leq K,$
and $V_{j}\subseteq \mathbb{R}^{d}$ are
open sets such that $\partial\Omega=\bigcup_{j=1}^{K}V_{j}$. Some measures are defined as
follows:
$$h_{I}:=\sup_{\textbf{y}\in\Omega}\min_{\textbf{y}_{j}\in Y^{I}}\|\textbf{y}-\textbf{y}_{j}\|_{2},\ q_{I}:=\frac{1}{2}\min_{\textbf{y}_{i},\textbf{y}_{k}\in Y^{I},i\neq k}\|\textbf{y}_{i}-\textbf{y}_{k}\|_{2},
\quad \textrm{for interior centers} \ Y^{I},$$
$$h_{B}:=\max_{1\leq j\leq K}h_{Bj}=\max_{1\leq j\leq K}\left\{\sup_{\textbf{y}\in B=B(\textbf{0},1)}
\min_{\textbf{y}_{j}\in \psi_{j}^{-1}\left(Y^{B}\bigcap V_{j}\right)}\|\textbf{y}-\textbf{y}_{j}\|_{2}\right\},
\quad \textrm{for boundary centers} \ Y^{B},$$
$$q_{B}:=\min_{1\leq j\leq K}q_{Bj}=\min_{1\leq j\leq K}\left\{\frac{1}{2}\min_{\textbf{y}_{i},\textbf{y}_{k}
\in \psi_{j}^{-1}\left(Y^{B}\bigcap V_{j}\right),i\neq k}\|\textbf{y}_{i}-\textbf{y}_{k}\|_{2}\right\},
\quad \textrm{for boundary centers} \ Y^{B},$$
$$h_{Y}:=\max\{h_{I},h_{B}\}<1,\quad q_{Y}:=\min\{q_{I},q_{B}\}.$$
Using the scaled radial function $\Phi_{\delta}$ as basis functions, then the radial basis functions trial space is
\begin{equation}
\label{ts}
W_{\delta}=\textrm{span}\{\Phi_{\delta}(\cdot-\textbf{y})|\textbf{y}\in Y\}.
\end{equation}

In the later theoretical analysis, we need two important auxiliary results: inverse inequality and sampling inequality.
The following inverse inequality has been proved for RBF approximation in the reference \cite{LIU23} on the bounded domains.

\begin{lemma}\label{invv}(\textbf{inverse inequality})
Let $\Omega\subseteq \mathbb{R}^{d}$ be a bounded domain satisfying an interior cone condtion. Let $Y$ be
quasi-uniform in the sense that $h_{Y}\leq cq_{Y}$. For all $\delta-$dependent functions of the form $f=\sum_{j=1}^{N}c_{j}\Phi_{\delta}(\cdot-\textbf{x}_{j}), \ \textbf{x}_{j}\in Y,$
there exists a constant $C$ such that
\begin{equation}\label{inv}
\|f\|_{W^{\sigma,2}(\Omega)}\leq C\delta^{\sigma}h_{Y}^{-\sigma+d/2}\|f\|_{L^{2}(\Omega)}.
\end{equation}
\end{lemma}

The well-known sampling inequality has been proved in
\cite{nar06,rieger10}.
\begin{lemma}\label{wl3}(\textbf{sampling inequality})
Suppose $\Omega\subseteq \mathbb{R}^{d}$ is a bounded domain with an interior cone condition. Let
$q\in[1,\infty]$ and constants $0\leq \mu<\mu+d/2<\lfloor m \rfloor$. Then there are positive constant $C$, $h_{0}$
such that
\begin{equation}
\|f\|_{W^{\mu,q}(\Omega)}\leq C\left(h_{Y}^{m-\mu-d(1/2-1/q)_{+}}\|f\|_{W^{m,2}(\Omega)}
+h_{Y}^{-\mu}\|\Pi_{Y}f\|_{\infty,Y}\right)
\end{equation}
holds for every discrete set $Y\subset\Omega$ with fill distance at most $h_{Y}\leq h_{0}$ and
every $f\in W^{m,2}(\Omega)$. Here $\Pi_{Y}$ is the function evaluation operator at the data sites $Y$.
\end{lemma}

\section{Meshfree discretization}

\subsection{Well-posedness}
There are several theories which describe the regularity and the well-posedness of the Monge-Amp$\grave{\textrm{e}}$re equation
\cite{caff,WANG,FIG}, for example the existence theorem from \cite{FIG}:
\begin{theorem}\label{exit}(\textbf{existence})
Let $\Omega$ be a uniformly convex domain with $C^{k+2,\gamma}-$boundary, $k\geq2,\gamma\in(0,1)$. Let $f\in C^{k,\gamma}(\bar{\Omega})$
be strictly positive. Then for any $g\in C^{k+2,\gamma}(\partial\Omega)$, there exist a unique solution $u^{*}\in C^{k+2,\gamma}(\bar{\Omega})$
to the Dirichlet problem (\ref{ma1})-(\ref{ma2}).
\end{theorem}
In the later discussion, we always assume that the conditions of Theorem \ref{exit} are satisfied.
And we assume that the solution of the equations (\ref{2ma1})-(\ref{2ma2}) has the higher regularity $u^{*}\in W^{\sigma,2}(\Omega)$
with $\sigma$ such that $\lfloor \sigma\rfloor>k+2+d/2$. Of course, $u^{*}\in C^{k+2,\gamma}(\overline{\Omega})$ according to Sobolev inequality from section 5.6.3 of \cite{evans}. In addition, $F$ is Fr$\acute{\textrm{e}}$chet differentiable in
$K_{R}(u^{*})=\{u\in W^{\sigma,2}(\Omega):\|u-u^{*}\|_{W^{\sigma,2}(\Omega)}\leq R\}\subset D(F)$, and
\begin{equation}
F'(u^{*})(v)=u^{*}_{yy}v_{xx}+u^{*}_{xx}v_{yy}-2u^{*}_{xy}v_{xy}.
\end{equation}
Obviously, the convex property of $u^{*}$ implies that the linear operator $F'(u^{*})$ is uniformly elliptic.
Thus, we can obtain a prior bounds of the fully nonlinear problem (\ref{2ma1})-(\ref{2ma2}) by utilizing the prior estimate of the linear elliptic problem
\begin{eqnarray}
 F'(u)v(x,y) &=& f(x,y), \quad\textrm{in}\ \Omega\subset \mathbb{R}^{2},\label{xian1}\\
  v(x,y) &=& g(x,y),\quad \textrm{on}\ \partial\Omega,\label{xian2}
\end{eqnarray}
here $u\in K_{R}(u^{*})$.
\begin{theorem}\label{pp}(\textbf{a priori bounds of (\ref{2ma1})-(\ref{2ma2})})
Suppose $\Omega\subseteq\mathbb{R}^{d}, d=2$ is a uniformly convex domain with
$C^{k+2,\gamma}-$boundary. Then the Frech$\acute{\textrm{e}}$t
derivative $F'(u)$ at $u$ be bounded and Lipschitz continuous, namely
\begin{equation}\label{lip}
\|F'(u)-F'(v)\|:=\|F'(u^{*})u-F'(u^{*})v\|_{L^{2}(\Omega)}\leq C''\|u-v\|_{W^{\sigma,2}(\Omega)},\quad \textrm{for all}\ u,v\in K_{R}(u^{*}).
\end{equation}
Assume the radius $R$ to be small enough to satisfy
\begin{equation}\label{ccon}
C_{l}C''R\leq\frac{1}{2},
\end{equation}
then for any $u,v\in K_{R}(u^{*})$, there exists two constants $C,C_{b}$ such that
\begin{equation}\label{xiajie}
\|u-v\|_{L^{2}(\Omega)}\leq C\{\|Fu-Fv\|_{L^{2}(\Omega)}+\|u-v\|_{W^{1/2,2}(\partial\Omega)}\},
\end{equation}
\begin{equation}\label{shangjie}
\|Fu-Fv\|_{W^{\sigma-2,2}(\Omega)}\leq C_{b}\|u-v\|_{W^{\sigma,2}(\Omega)}.
\end{equation}
\end{theorem}
\begin{proof}
The inequality (\ref{lip}) is obvious because $\|F'(u)-F'(v)\|$ can be bounded by
\begin{align*}
\|F'(u^{*})u-F'(u^{*})v\|_{L^{2}(\Omega)}&=
\|u^{*}_{xx}(u-v)_{yy}+u^{*}_{yy}(u-v)_{xx}-2u^{*}_{xy}(u-v)_{xy}\|_{L^{2}(\Omega)}\\[4pt]
&\leq |u^{*}|_{W^{2,2}(\Omega)}|u-v|_{W^{2,2}(\Omega)}\\[4pt]
&\leq C^{''}\|u-v\|_{W^{\sigma,2}(\Omega)}.
\end{align*}
With the help of the priori inequality of linear elliptic problems (for example Corollary 8.7 of \cite{gt},), we have
\begin{equation*}
\|u-v\|_{L^{2}(\Omega)}\leq C_{l}\{\|F'(u)(u-v)\|_{L^{2}(\Omega)}+\|u-v\|_{W^{1/2,2}(\partial\Omega)}\}.
\end{equation*}
Then by using Taylor formula and Lipschitz continuity of $F'(u)$, we have
\begin{align*}
\|u-v\|_{L^{2}(\Omega)}&\leq C_{l}\{\|F'(u)(u-v)-Fu+Fv\|_{L^{2}(\Omega)}+\|Fu-Fv\|_{L^{2}(\Omega)}
+\|u-v\|_{W^{1/2,2}(\partial\Omega)}\}\\[4pt]
&\leq C_{l}\{\sup_{w\in\overline{uv}}\|F'(w)-F'(v)\|\|u-v\|_{L^{2}(\Omega)}+\|Fu-Fv\|_{L^{2}(\Omega)}
+\|u-v\|_{W^{1/2,2}(\partial\Omega)}\}\\[4pt]
&\leq C_{l}\{C''\|u-v\|_{W^{\sigma,2}(\Omega)}\|u-v\|_{L^{2}(\Omega)}+\|Fu-Fv\|_{L^{2}(\Omega)}
+\|u-v\|_{W^{1/2,2}(\partial\Omega)}\}\\[4pt]
&\leq C_{l}\{C''R\|u-v\|_{L^{2}(\Omega)}+\|Fu-Fv\|_{L^{2}(\Omega)}
+\|u-v\|_{W^{1/2,2}(\partial\Omega)}\}.
\end{align*}
This yields the result (\ref{xiajie}). In addition,
\begin{align*}
\|Fu-Fv\|_{W^{\sigma-2,2}(\Omega)}&\leq \|Fv-Fu-F'(u)(v-u)\|_{W^{\sigma-2,2}(\Omega)}+\|F'(u)(v-u)\|_{W^{\sigma-2,2}(\Omega)}\\[4pt]
&\leq C''R \|u-v\|_{W^{\sigma,2}(\Omega)}+(C''R+\|F'(u^{*})\|) \|u-v\|_{W^{\sigma,2}(\Omega)}\\[4pt]
&\leq (2C''R+\|F'(u^{*})\|)\|u-v\|_{W^{\sigma,2}(\Omega)}.
\end{align*}
Setting $C_{b}=2C''R+\|F'(u^{*})\|$ yields (\ref{shangjie}).
\end{proof}
\subsection{Collocation}
We are now consider the meshfree collocation for the Monge-Amp$\grave{\textrm{e}}$re equation (\ref{2ma1})-(\ref{2ma2}).
Denote $\textbf{x}=(x,y)$ and $\textbf{x}_{i}=(x_{i},y_{i})$, we use a set $X=X^{I}\cup X^{B}$ with
$X^{I}=\{\textbf{x}_{1},\textbf{x}_{2},\ldots,\textbf{x}_{M_{I}}\}\subseteq \Omega$ and $X^{B}=\{\textbf{x}_{M_{I}+1},\textbf{x}_{M_{I}+2},\ldots,\textbf{x}_{M}\}\subseteq \partial\Omega$ as collocation points,
and use
$$s_{I}:=\sup_{\textbf{x}\in\Omega}\min_{\textbf{x}_{j}\in X^{I}}\|\textbf{x}-\textbf{x}_{j}\|_{2},$$
$$s_{B}:=\max_{1\leq j\leq K}s_{Bj}=\max_{1\leq j\leq K}\left\{\sup_{\textbf{x}\in B=B(\textbf{0},1)}
\min_{\textbf{x}_{j}\in \psi_{j}^{-1}\left(X^{B}\bigcap V_{j}\right)}\|\textbf{x}-\textbf{x}_{j}\|_{2}\right\}$$
to represent the fill distance of the data sites $X^{I}$ and $X^{B}$ respectively. Obviously,
they have the parallel definitions with $h_{I}$ and $h_{B}$ on trial side. Hence choose $s_{X}=\max\{s_{I},s_{B}\}<1$.
Then we can construct a approximate function with the form
$$s(\textbf{x})=\sum_{j=1}^{N}c_{j}\Phi_{\delta}(\textbf{x}-\textbf{y}_{j}),\quad \textbf{y}_{j}\in Y=Y^{I}\cup Y^{B}\subset\mathbb{R}^{2},$$
which satisfies
\begin{align}
Fs(x_{i},y_{i}) & = f(x_{i},y_{i}),\quad (x_{i},y_{i})\in X^{I},\quad i=1,\ldots,M_{I},
\label{ls1} \\[4pt]
s(x_{i},y_{i}) & = g(x_{i},y_{i}), \quad (x_{i},y_{i})\in X^{B},\quad i=M_{I+1},\ldots,M.
\label{ls2}
\end{align}
(\ref{ls1})-(\ref{ls2}) is a meshfree discretization of the equation (\ref{2ma1})-(\ref{2ma2}),
but with a nonsquare nonlinear systems. Let $\Pi_{X^{I}}$ and $\Pi_{X^{B}}$ be
the function evaluation operators at the data sites $X^{I}$ and $X^{B}$ respectively,
then the nonlinear systems (\ref{ls1})-(\ref{ls2}) can also be rewritten as
\begin{align}
\Pi_{X^{I}}Fs& =\Pi_{X^{I}}f = \Pi_{X^{I}}Fu^{*},
\label{61} \\[4pt]
\Pi_{X^{B}}s & = \Pi_{X^{B}}g =\Pi_{X^{B}}u^{*}.
\label{62}
\end{align}
Because the nonlinear systems (\ref{61})-(\ref{62}) is overdetermined, thus we only require that one of the nonlinear numerical methods is clever enough to produce a $s\in W_{\delta}\cap K_{R}(u^{*})$ with the given tolerance
\begin{equation}\label{tolerance}
\max\left\{\|\Pi_{X^{I}}(Fu^{*}-Fs)\|_{\infty,X^{I}},\|\Pi_{X^{B}}(u^{*}-s)\|_{\infty,X^{B}}\right\}\leq \emph{tol}.
\end{equation}
This can be done by residual minimization
\begin{equation}\label{residual}
s=\arg\min\left\{\|\Pi_{X^{I}}(Fu^{*}-Fs)\|_{\infty,X^{I}},\|\Pi_{X^{B}}(u^{*}-s)\|_{\infty,X^{B}}:
s\in  W_{\delta}\cap K_{R}(u^{*})\right\}.
\end{equation}
\section{Convergence analysis}
The following theorem is our main theoretical result which proves the convergence rate of meshfree method for the
Monge-Amp$\grave{\textrm{e}}$re equation (\ref{2ma1})-(\ref{2ma2}).
\begin{theorem}
Let $\Omega\subseteq \mathbb{R}^{d}, d=2$ be a uniformly convex domain has a $C^{k+2,\gamma}-$boundary with $\gamma\in[0,1)$.
In the equation (\ref{2ma1})-(\ref{2ma2}), let $f\in C^{k,\gamma}(\bar{\Omega})$
be strictly positive and $g\in C^{k+2,\gamma}(\partial\Omega)$. We denote the unique solution of (\ref{2ma1})-(\ref{2ma2}) as $u^{*}\in W^{\sigma,2}(\Omega)$, $\lfloor \sigma\rfloor>k+2+d/2$.
Assume the radius $R$ of the ball $K_{R}$ to be small enough to satisfy (\ref{ccon}).
Let $Y$ be
quasi-uniform in the sense that $h_{Y}\leq cq_{Y}$. Let the collocation sites $X$ also be quasi-uniform.
Let $s_{I}$ and $s_{B}$ be the fill distance of the data sites $X^{I}$ and $X^{B}$
respectively, and $s_{X}=\max\{s_{I},s_{B}\}$. Suppose $s\in W_{\delta}$ is the
collocation solution of (\ref{ls1})-(\ref{ls2}) obtained by choosing
\begin{equation}\label{tolerance22}
tol= C\cdot s^{1/2}_{B}\delta^{-\sigma}h_{Y}^{\sigma-2-d/2}\|u^{*}\|_{W^{\sigma,2}(\Omega)}.
\end{equation}
If the testing discretizations is finer than the trial discretization and
satisfies condition
\begin{equation}\label{condition}
C\cdot C_{b}\cdot \delta^{\sigma}s_{X}^{\sigma-2}h_{Y}^{-\sigma+d/2}<\frac{1}{2},
\end{equation}
then there exists a constant $C>0$ such that
\begin{equation*}\label{46}
\|u^{*}-s\|_{L^{2}(\Omega)}\leq
C\delta^{-\sigma}h^{\sigma-3}_{Y}\|u^{*}\|_{W^{\sigma,2}(\Omega)}.
\end{equation*}
\end{theorem}
\begin{proof}
We want to bound the $L^{2}-$error between the solution $u^{*}$ and its approximation $s$. We can use the radial basis functions  interpolant as a transition. The interpolant is a map $I_{h}: u^{*}\rightarrow I_{h}u^{*}\in W_{\delta}$ defined by
\begin{equation*}\label{interpolation}
I_{h}u^{*}(\textbf{y}_{k})=\sum_{j=1}^{N}c_{j}\Phi_{\delta}(\textbf{y}_{k}-\textbf{y}_{j})=u^{*}(\textbf{y}_{k}),\quad \textbf{y}_{k}\in Y\subset\mathbb{R}^{2}, k=1,2,\ldots,N,
\end{equation*}
for the unknown coefficients $c_{1},c_{2},\ldots,c_{N}$ are determined by solving the above linear system. Then the $L^{2}-$error between the solution $u^{*}$ and its approximation $s$ can be split into
\begin{align*}
\|u^{*}-s\|_{L^{2}(\Omega)}&\leq \|u^{*}-I_{h}u^{*}\|_{L^{2}(\Omega)}+\|I_{h}u^{*}-s\|_{L^{2}(\Omega)}.
\end{align*}
The first term on the right-hand side is the radial basis functions interpolation error, which can be easily obtained by using the sampling inequality (Lemma \ref{wl3}) and the stability of $I_{h}$ (Lemma 5.1 of \cite{LIU23}), namely
\begin{align}\label{zz}
\|u^{*}-I_{h}u^{*}\|_{L^{2}(\Omega)}\leq Ch_{Y}^{\sigma}\|u^{*}-I_{h}u^{*}\|_{W^{\sigma,2}(\Omega)}\leq
 C\delta^{-\sigma}h_{Y}^{\sigma}\|u^{*}\|_{W^{\sigma,2}(\Omega)}.
\end{align}
Hence, our main task in this proof is to bound the second term on the right-hand side.

Choosing $u=I_{h}u^{*},v=s$ and squaring both sides of (\ref{xiajie}), we have
\begin{equation*}
\|I_{h}u^{*}-s\|_{L^{2}(\Omega)}^{2}\leq 2C\{\|F(I_{h}u^{*})-Fs\|_{L^{2}(\Omega)}^{2}+\|I_{h}u^{*}-s\|_{W^{1/2,2}(\partial\Omega)}^{2}\}.
\end{equation*}
Then we can bound the two terms on the right-hand side separately. For the first term, using Lemma \ref{wl3},
the inequality (\ref{shangjie}), Lemma \ref{invv}, and the stability of radial basis functions interpolation (Lemma 5.1 of \cite{LIU23}) yields
\begin{align*}
\|F(I_{h}u^{*})-Fs\|_{L^{2}(\Omega)}^{2}&\leq C\left\{s_{I}^{2\sigma-4}\|F(I_{h}u^{*})-Fs\|^{2}_{W^{\sigma-2,2}(\Omega)}
+\|\Pi_{X^{I}}(F(I_{h}u^{*})-Fs)\|^{2}_{\infty,X^{I}}\right\}
\\[4pt]
&\leq C\cdot C_{b}^{2}\cdot s_{I}^{2\sigma-4}\|I_{h}u^{*}-s\|^{2}_{W^{\sigma,2}(\Omega)}
+C\|\Pi_{X^{I}}(F(I_{h}u^{*})-Fu^{*})\|^{2}_{\infty,X^{I}}+C\cdot tol^{2}
\\[4pt]
&= C\cdot C_{b}^{2}\cdot s_{I}^{2\sigma-4}\|I_{h}u^{*}-s\|^{2}_{W^{\sigma,2}(\Omega)}
+C\|F(I_{h}u^{*})-Fu^{*}\|^{2}_{\infty,X^{I}}+C\cdot tol^{2}
\\[4pt]
&\leq C\cdot C_{b}^{2}\cdot\delta^{2\sigma}s_{I}^{2\sigma-4}h_{Y}^{-2\sigma+d}\|I_{h}u^{*}-s\|^{2}_{L^{2}(\Omega)}
+C\|F(I_{h}u^{*})-Fu^{*}\|^{2}_{\infty,\Omega}+C\cdot tol^{2}
\\[4pt]
&\leq C\cdot C_{b}^{2}\cdot\delta^{2\sigma}s_{X}^{2\sigma-4}h_{Y}^{-2\sigma+d}\|I_{h}u^{*}-s\|^{2}_{L^{2}(\Omega)}
+C\|I_{h}u^{*}-u^{*}\|^{2}_{W^{2,\infty}(\Omega)}+C\cdot tol^{2}\\[4pt]
&\leq C\cdot C_{b}^{2}\cdot\delta^{2\sigma}s_{X}^{2\sigma-4}h_{Y}^{-2\sigma+d}\|I_{h}u^{*}-s\|^{2}_{L^{2}(\Omega)}
+Ch_{I}^{2(\sigma-2-d/2)}\|I_{h}u^{*}-u^{*}\|^{2}_{W^{\sigma,2}(\Omega)}+C\cdot tol^{2}\\[4pt]
&\leq C\cdot C_{b}^{2}\cdot\delta^{2\sigma}s_{X}^{2\sigma-4}h_{Y}^{-2\sigma+d}\|I_{h}u^{*}-s\|^{2}_{L^{2}(\Omega)}
+C\delta^{-2\sigma}h_{I}^{2(\sigma-2-d/2)}\|u^{*}\|^{2}_{W^{\sigma,2}(\Omega)}+C\cdot tol^{2}\\[4pt]
&\leq C\cdot C_{b}^{2}\cdot\delta^{2\sigma}s_{X}^{2\sigma-4}h_{Y}^{-2\sigma+d}\|I_{h}u^{*}-s\|^{2}_{L^{2}(\Omega)}
+C\delta^{-2\sigma}h_{Y}^{2(\sigma-2-d/2)}\|u^{*}\|^{2}_{W^{\sigma,2}(\Omega)}+C\cdot tol^{2}\\[4pt]
&\leq C\cdot C_{b}^{2}\cdot\delta^{2\sigma}s_{X}^{2\sigma-4}h_{Y}^{-2\sigma+d}\|I_{h}u^{*}-s\|^{2}_{L^{2}(\Omega)}
+C\delta^{-2\sigma}h_{Y}^{2(\sigma-2-d/2)}\|u^{*}\|^{2}_{W^{\sigma,2}(\Omega)}.
\end{align*}

For the boundary term, we denote $w_{j}=((I_{h}u^{*}-s)\omega_{j})\circ\Psi_{j}\in W^{\sigma-1/2,2}(B)$.
From Lemma \ref{wl3}, the trace theorem (Theorem 8.7 of \cite{wl87}), the Sobolev imbedding theorems (Theorem 4.12 of \cite{adams}),
the quasi-uniform property of $X$, and the quasi-uniform property of $Y$,we can bound
\begin{align*}
&\quad\|I_{h}u^{*}-s\|^{2}_{W^{1/2,2}(\partial\Omega)}=\sum_{j=1}^{K}\|w_{j}\|^{2}_{W^{1/2,2}(B)}
\\[4pt]
&\leq C\sum_{j=1}^{K}\left(s_{Bj}^{\sigma-1}\|w_{j}\|_{W^{\sigma-1/2,2}(B)}+
s_{Bj}^{-1/2}\|\Pi_{\psi_{j}^{-1}\left(X^{B}\bigcap V_{j}\right)} w_{j}\|_{\infty,\psi_{j}^{-1}\left(X^{B}\bigcap V_{j}\right)}
\right)^{2}
\\[4pt]
&\leq C\sum_{j=1}^{K}\left(s_{Bj}^{2\sigma-2}\|w_{j}\|^{2}_{W^{\sigma-1/2,2}(B)}+
s_{Bj}^{-1}\|\Pi_{\psi_{j}^{-1}\left(X^{B}\bigcap V_{j}\right)} w_{j}\|^{2}_{\infty,\psi_{j}^{-1}\left(X^{B}\bigcap V_{j}\right)}
\right)
\\[4pt]
&\leq Cs_{B}^{2\sigma-2}\|I_{h}u^{*}-s\|^{2}_{W^{\sigma-1/2,2}(\partial\Omega)}+Cs_{B}^{-1}
\|\Pi_{X^{B}} (I_{h}u^{*}-s)\|^{2}_{\infty,X^{B}}
\\[4pt]
&\leq Cs_{B}^{2\sigma-2}\|I_{h}u^{*}-s\|^{2}_{W^{\sigma,2}(\Omega)}+C s_{B}^{-1}
\|\Pi_{X^{B}} (I_{h}u^{*}-u^{*})\|^{2}_{\infty,X^{B}}+Cs_{B}^{-1}\cdot tol^{2}\\[4pt]
&\leq C\delta^{2\sigma}s_{X}^{2\sigma-2}h_{Y}^{-2\sigma+d}\|I_{h}u^{*}-s\|^{2}_{L^{2}(\Omega)}+C s_{B}^{-1}
\|I_{h}u^{*}-u^{*}\|^{2}_{\infty,X^{B}}+Cs_{B}^{-1}\cdot tol^{2}\\[4pt]
&\leq C\delta^{2\sigma}s_{X}^{2\sigma-2}h_{Y}^{-2\sigma+d}\|I_{h}u^{*}-s\|^{2}_{L^{2}(\Omega)}+Cs_{B}^{-1}\max_{1\leq j\leq K}\|w_{j}\|^{2}_{L^{\infty}(B)}+Cs_{B}^{-1}\cdot tol^{2}
\\[4pt]
&\leq C\delta^{2\sigma}s_{X}^{2\sigma-2}h_{Y}^{-2\sigma+d}\|I_{h}u^{*}-s\|^{2}_{L^{2}(\Omega)}+Cs_{B}^{-1}\max_{1\leq j\leq K}h_{Bj}^{2(\sigma-1/2-d/2)}\|w_{j}\|^{2}_{W^{\sigma-1/2,2}(B)}+Cs_{B}^{-1}\cdot tol^{2}
\\[4pt]
&\leq C\delta^{2\sigma}s_{X}^{2\sigma-2}h_{Y}^{-2\sigma+d}\|I_{h}u^{*}-s\|^{2}_{L^{2}(\Omega)}+Cs_{B}^{-1} h_{B}^{2(\sigma-1/2-d/2)}\sum_{j=1}^{K}\|w_{j}\|^{2}_{W^{\sigma-1/2,2}(B)}+Cs_{B}^{-1}\cdot tol^{2}\\[4pt]
&=C\delta^{2\sigma}s_{X}^{2\sigma-2}h_{Y}^{-2\sigma+d}\|I_{h}u^{*}-s\|^{2}_{L^{2}(\Omega)}+Cs_{B}^{-1} h_{B}^{2(\sigma-1/2-d/2)}\|I_{h}u^{*}-u^{*}\|^{2}_{W^{\sigma-1/2,2}(\partial\Omega)}+Cs_{B}^{-1}\cdot tol^{2}\\[4pt]
&\leq
C\delta^{2\sigma}s_{X}^{2\sigma-2}h_{Y}^{-2\sigma+d}\|I_{h}u^{*}-s\|^{2}_{L^{2}(\Omega)}+Cs_{B}^{-1} h_{B}^{2(\sigma-1/2-d/2)}\|I_{h}u^{*}-u^{*}\|^{2}_{W^{\sigma,2}(\Omega)}+Cs_{B}^{-1}\cdot tol^{2}
\end{align*}
\begin{align*}
&\leq C\delta^{2\sigma}s_{X}^{2\sigma-2}h_{Y}^{-2\sigma+d}\|I_{h}u^{*}-s\|^{2}_{L^{2}(\Omega)}+C\delta^{-2\sigma} h_{Y}^{-(\sigma-d/2)/(\sigma-2)}h_{B}^{2(\sigma-1/2-d/2)}\|u^{*}\|^{2}_{W^{\sigma,2}(\Omega)}
+Cs_{B}^{-1}\cdot tol^{2}\\[4pt]
&\leq C\delta^{2\sigma}s_{X}^{2\sigma-2}h_{Y}^{-2\sigma+d}\|I_{h}u^{*}-s\|^{2}_{L^{2}(\Omega)}+
C\delta^{-2\sigma}h_{Y}^{2(\sigma-2-d/2)}\|u^{*}\|^{2}_{W^{\sigma,2}(\Omega)}.
\end{align*}
By using the condition (\ref{condition}), combining both above results with the interpolation error (\ref{zz}) and setting $d=2$, we obtain the convergence rate.
\end{proof}

\section{Conclusions}
Although the meshfree method has been used for solving the Monge-Amp\textrm{$\grave{\textrm{e}}$}re equation for a long time,
we proved the convergence of this method until now in this paper.
The theoretical convergence was established by using the regularity theory of the solution,
the inverse inequality, the sampling inequality, and the approximation property of the radial basis functions trial spaces.
However, the convergence rate is currently not yet optimal and deserve refinement.

\bibliographystyle{model1c-num-names}
\bibliography{<your-bib-database>}

\end{document}